\documentclass{article}
\usepackage[utf8]{inputenc}
\usepackage{amsmath, amsthm, amssymb, amsfonts, comment}
\usepackage{bm}
\usepackage{dsfont}
\usepackage[utf8]{inputenc}
\usepackage{rotate}
\usepackage{tikz}
\usepackage{tikz-cd}
\usepackage[arrow,matrix,curve]{xy} 	
\usepackage{xcolor}
\usepackage{todonotes}
\usepackage{alltt} 
\usepackage{url}

\theoremstyle{plain}
\newtheorem{theorem}{Theorem}[section]

\newtheorem{lemma}[theorem]{Lemma}

\theoremstyle{definition}

\newtheorem{example}[theorem]{Example}

\newcommand{\norm}[1]{\left\lVert#1\right\rVert}

\title{P-adic Poissonian Pair Correlations via the Monna Map}
\author{Christian Weiss}
\date{\today}

\begin{document}

\maketitle

\begin{abstract} 
Although the existence of sequences in the p-adic integers with Poissonian pair correlations has already been shown, no explicit examples had been found so far. In this note we discuss how to transfer real sequences with Poissonian pair correlations to the p-adic setting by making use of the Monna map.
\end{abstract}

\section{Introduction}
The concept of Poissonian pair correlations in $\mathbb{R}$ was popularized by Rudnick and Sarnak in \cite{RS98} and has attracted significant attention since. A sequence $(x_n)_{n \in \mathbb{N}} \subset [0,1]$ is said to have $\alpha$-Poissonian pair correlations for $0 < \alpha \leq 1$ if
$$F_{N,\alpha}(s) := \frac{1}{N^{2-\alpha}} \# \left\{ 1 \leq k \neq l \leq N \ : \ \left\| x_k - x_l\right\|_{\infty} \leq \frac{s}{N^\alpha} \right\},$$
where $\left\| \cdot \right\|$ is the distance of a number from its nearest integer, converges to $2s$ for all $s \geq 0$. Thus it describes a special behaviour of gaps on a local scale and is a generic property of random sequences in $[0,1)$ drawn from uniform distribution, see e.g. \cite{HKL19}. Probably the most famous explicit example of a sequence in $[0,1)$ with Poissoinan pair correlations for $\alpha = 1$, or short Poissoinan pair correlations, is the sequence $(x_n)_{n \in \mathbb{N} \setminus \Box} = \{ \sqrt{n}\}$, where $\Box$ denotes the set of perfect squares and $\{ \cdot \}$ is the fractional part of a number, see \cite{BMV15}. Other examples have been constructed e.g. in \cite{LST24}.\\[12pt]
Poissonian pair correlations may not only be defined for sequences of one-dimensional real numbers but were among others generalized to higher dimension, see \cite{HKL19}, higher orders, see \cite{HZ23}. and the p-adic integers $\mathbb{Z}_p$, see \cite{Wei23}. In this paper we are interested in the p-adic case.\\[12pt]
For a prime-number $p \in \mathbb{Z}$, the p-adic numbers $\mathbb{Q}_p$ are the completion of $\mathbb{Q}$ with respect to the p-adic absolute value$ |\cdot|_p$, see Section~\ref{sec:proofs} for more details. The p-adic integers are defined as
$$\mathbb{Z}_p=\left\{ x \in \mathbb{Q}_p \, : \, \left|x\right|_p \leq 1 \right\} \subset \mathbb{Q}_p$$
and may thus be regarded as the p-adic analogue of $[0,1) \subset \mathbb{R}.$
For any $z \in \mathbb{Z}_p$ and $s \in \mathbb{R}_0^+$, the disc of radius $s$ around $z$  is defined by
$$D_p(z,s) := \left\{ x \in \mathbb{Z}_p \, : \, \left| x - z \right|_p \leq s \right\} \subset \mathbb{Z}_p$$ 
Note that the p-adic absolute value can only take values $p^{-k}$ and thus $$D_p(z,s) = D_p(z,p^{-k_0})$$ for the smallest $k_0$ with $p^{-k_0} \leq s$. For the Haar measure on $\mathbb{Q}_p$ we thus have $\mu(D_p(z,s)) = \mu(D_p(z,p^{-k_0})) = p^{-k_0}.$\\[12pt]
According to \cite{Wei23} p-adic Poissonian pair correlations for $0 < \alpha \leq 1$ are defined in the following way: let $(x_n)_{n \in \mathbb{N}} \subset \mathbb{Z}_p$ be a sequence and let
$$F_{N,\alpha,p}(s) := \frac{1}{N^2} \frac{1}{\mu\left( D_p(0,s/N^\alpha)\right)} \# \left\{ 1 \leq i \neq j \leq N \, : \, \left|x_i-x_j\right|_p \leq \frac{s}{N^\alpha} \right\}.$$
We say that $(x_n)_{n\in \mathbb{N}} \subset \mathbb{Z}_p$ has $\alpha$-Poissonian pair correlations for $0 \leq \alpha \leq 1$ if 
$$\lim_{N \to \infty} F_{N,\alpha,p}(s) = 1$$ 
for all $s \geq 0$. If $\alpha = 1$, then we also just speak of Poissonian pair correlations. The necessity to normalize the limit to $1$ is essentially caused by the jumps in $\mu(D_p(0,s/N^\alpha))$ whenever the radius decreases below a new negative power of the prime $p$.\\[12pt]
In \cite{Wei23}, Theorem 1.4, it was shown that a sequence of independent random variables, which are uniformly distributed on $\mathbb{Z}_p$, almost surely has $\alpha$-Poissonian pair correlations for any $0 < \alpha \leq 1$. Moreover the simple sequence $(n)_{n \in \mathbb{N}}$ was identified as an example with $\alpha$-Poissonian for all $0 < \alpha < 1$ but fails to have Poissonian pair correlations. No explicit example has been found in the case $\alpha = 1$ so far and it was asked in \cite{Wei23} how to construct any. Even worse, it was shown in \cite{Wei24}, Theorem 1.6, that there does not exist any sequence $(x_n)_{n \in \mathbb{N}} = (f(n))_{n \in \mathbb{N}} \subset \mathbb{Z}_p$, where $f$ is a polynomial, with Poissonian pair correlations.\\[12pt]
In this note, we show how to transfer any sequence $(x_n)_{n \in \mathbb{N}} \subset [0,1)$ with $\alpha$-Poissonian pair correlations into a sequence $(y_n)_{n \in \mathbb{N}} \subset \mathbb{Z}_p$ with the same property in the p-adic sense. For this purpose the so-called Monna  map $\varphi_p: \mathbb{Z}_p \to \mathbb{R}$, which was introduced in \cite{Mon52}, serves as the main tool. Let
$$x = \sum_{i=0}^\infty a_i p^i$$
with $0 \leq a_i \leq p-1$ be a p-adic number. Then the Monna map is defined as
\begin{equation} \label{eq:vdc}
\varphi(x) = y = \sum_{i=0}^\infty a_i p^{-i-1}.
\end{equation}
Evidently, the map is linked to the definition of van der Corput sequences known from discrepancy theory, see e.g. \cite{Mei67,Wei24}. While $\varphi$ is obviously surjective, it is not injective on all of its domain because the right hand side of \eqref{eq:vdc} is not unique due to the limit of the geometric series. Nonetheless, an inverse map $\varphi^{-1}$ exists on the set
$$\left\{ x \in [0,1) \, : \, x = \sum_{i=0}^\infty a_i p^{-i-1}, \ a_i \neq p-1 \, \textrm{for infinitely many} \ i \right\}.$$
Moreover, we know that the inverse image of $x$ under $\varphi$ contains the element
$$y = \sum_{i=0}^\infty a_i p^{i-1}$$
because
$$\varphi(\sum_{i=0}^\infty a_i p^{i-1}) = \sum_{i=0}^\infty a_i p^{-(i-1)-1} = \sum_{i=0}^\infty a_i p^{-i}$$
and we may thus define $\varphi^{-1}(x):=y$ (of course $\varphi^{-1}$ is not the inverse function of $\varphi$, but we do not think this slight abuse of notation is misleading). Special care has to be taken about the interaction of the Monna map with the p-adic absolute value. Further details are described in Lemma~\ref{lem:monna} and the discussion thereafter. Hence, $\varphi^{-1}$ is defined uniquely and the map $\varphi^{-1}$ has the following transference property which is the main result of this note. 
\begin{theorem} \label{thm:transfer} Let $0 < \alpha \leq 1$. If $y_k = \varphi^{-1}(x_k) \in [0,1)$ has $\alpha$-Poissonian pair correlations, then $x_k$ has $\alpha$-p-adic Poissonian pair correlations. 
\end{theorem}
It is possible to immediately derive examples of sequences with p-adic $\alpha$-Poissonian pair correlations from it.
\begin{example} Choose a prime basis $p$ and let $(x_n)_{n \in \mathbb{N}}$ be the van der Corput sequence in base $b$, which is defined via $x_n=\varphi(n)$ or put differently $n = \varphi^{-1}(x_n)$. It is well-known that the van der Corput sequence in base $p$ has $\alpha$-Poissonian pair correlations for any $0 < \alpha < 1$, see e.g. \cite{Wei22a}. Hence $(x_n) = n \in \mathbb{Z}_p$ has $\alpha$-Poissonian pair correlations as well by Theorem~\ref{thm:transfer}. This fact was proven differently in \cite{Wei23} as well. However, neither the van der Corput sequence has real Poissonian pair correlations nor does the sequence $(n)_{n \in \mathbb{N}} \subset \mathbb{Z}_p$ have p-adic ones.
\end{example}
\begin{example} According to \cite{BMV15}, the sequence $(x_n)_{n \in \mathbb{N} \setminus \Box} = (\{ \sqrt{n}\})_{n \in \mathbb{N}}$ of real numbers has Poissonian pair correlations. Then $(\varphi^{-1}(x_n))_{n \in \mathbb{N}} \subset \mathbb{Z}_p$ is a sequence of p-adic integers with Poissonian pair correlations by Theorem~\ref{thm:transfer}. For $p=3$ we have for example
\begin{align*}
\varphi^{-1}(\{\sqrt{2}\}) & = 52102 + \sum_{k=10}^\infty a_kp^{k},\\
\varphi^{-1}(\{\sqrt{3}\}) & = 58142 + \sum_{k=10}^\infty a_kp^{k},\\
\varphi^{-1}(\{\sqrt{5}\}) & = 33081 + \sum_{k=10}^\infty a_kp^{k}.
\end{align*}
In Figure~\ref{fig:plot}, the plot of the p-adic pair correlation function $F_{N,1,3}(s)$ is plotted for $s \in \{0.1, 0.25, 0.5, 1, 2\}$ and $1 \leq N \leq 5.000$. The convergence of the function towards $1$ becomes plausible despite that there are big deviations from the limit for small $N$. As in the real case, the speed of convergence is not very fast and even for $N=5.000$ the values of $F_{5.000,1,3}(s)$ still vary between $0.98$ and $1.13$. Moreover, note that the visible jumps in the plot occur exactly when $s/N$ passes a new negative power of $p=3$. 
\end{example}
\begin{figure}[!ht]
     \centering
         \centering\includegraphics[width=1\textwidth]{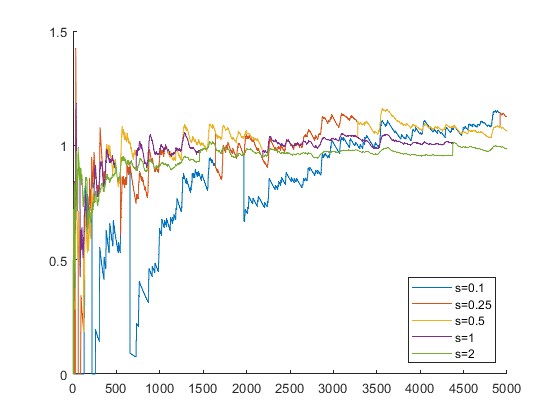}
         \caption{The p-adic pair correlation function $F_{N,1,3}(s)$ of $\varphi^{-1}(\{\sqrt{n}\})$ for $1 \leq N \leq 5.000$ and $s \in \{0.1, 0.25, 0.5, 1, 2\}$.}
         \label{fig:plot}
\end{figure}
The idea to use the Monna map to transfer properties from real uniform distribution theory to the p-adic setting has already been applied in \cite{Mei67} in the context of discrepancy theory, where the author proved the following.
\begin{theorem}[\cite{Mei67}, Theorem 4] A sequence $(x_n)_{n \in \mathbb{N}}$ is uniformly distributed in $\mathbb{Z}_p$ if and only if the real sequence $\varphi(x_n)$ is uniformly distributed in $[0,1)$.    
\end{theorem}
In the light of Meijer's result we ask, if $\varphi(x_n)_{n \in \mathbb{N}}$ has (real) Poissoian pair correlations if $(x_n)_{n \in \mathbb{N}}$ has this property in the p-adic sense. We leave this question open for future research.  As our proof of Theorem~\ref{thm:transfer} fundamentally relies on the monotonicity of $F_{N,\alpha}(s)$ in $s$ (Lemma~\ref{lem:convergence}) this question cannot be answered by simply following the lines of the proof presented here and reversing the direction of the arguments.
\section{A Transference Principle} \label{sec:proofs}
We shortly recall some basics about the p-adic numbers, which can be found e.g. in \cite{Neu99}. Let $p \in \mathbb{Z}$ be a prime number. For $a = \frac{b}{c}$ with $b,c \in \mathbb{Z} \setminus \{ 0 \}$, let $m$ be the highest possible power with $a = p^m \frac{b'}{c'}$ and $(b'c',p)=1$. The p-adic absolute value is defined as
$$|a|_p:= p^{-m}.$$
The field $\mathbb{Q}_p$ is the completion of $\mathbb{Q}$ with respect to $|\cdot|_p$. The p-adic integers are the closure of $\mathbb{Z}$ in this field. As $\mathbb{Z}_p = \{ x \in \mathbb{Q}_p \, : \, |x| \leq 1 \} \subset \mathbb{Q}_p$, this set is the p-adic analogue of $[0,1) \subset \mathbb{R}$. Finally, the ring of units $\mathbb{Z}_p^\times$ is given by $\mathbb{Z}_p^\times:=\left\{ y \in \mathbb{Z}_p \, : \, \left|y\right|_p = 1 \right\}$.\\[12pt]
For our applications, it is essential to understand the interaction of the Monna map $\varphi$ and the p-adic absolute value. Although the claims are essentially already contained in \cite{Mei67} and are certainly known to mathematicians familiar with this field, we nonetheless include the short proofs here for this reason.
\begin{lemma} \label{lem:monna} Let $x,y \in \mathbb{Z}_p$.
\begin{itemize} 
\item[(i)] The inequality $|x|_p \leq p^{-k}$ implies $|\varphi(x)| \leq p^{-k}$. 
\item[(ii)] On the other hand if $|\varphi(x))| \leq p^{-k}$, then $|x|_p \leq p^{-k-1}$. Moreover we either have $|x|_p \leq p^{-k}$ or $x = p^{k-1}$.
\item[(iii)] If $\varphi(x) = \varphi(y)$ then $x=y$ or
$$x = a_0 + a_1p + \ldots + a_{k-1}p^{k-1}$$
with $a_{k-1}>0$ and
$$y = a_0 + a_1p + \ldots + (a_{k-1}-1)p^{k-1}+\sum_{i=k}^\infty (p-1)p^i.$$
\end{itemize}
\end{lemma}
It is important to note that $|\varphi(a)| \neq|a|_p$ can thus hold because this is among the main reasons why the arguments in the proof of Theorem~\ref{thm:transfer} cannot be simply reverted to obtain the opposite implication.
\begin{proof}
(i) Let $x \in \mathbb{Z}_p$ with $|x|_p \leq p^{-k}$. By the definition of the p-adic absolute value this implies $x = a_kp^k + a_{k+1}p^{k+1} + \ldots$ with $0 \leq a_k \leq p-1$. Therefore,
\begin{align*}
|\varphi(x)| & = a_k p^{-k-1} + a_{k+1}p^{-k-2} + \ldots\\
&\leq (p-1) p^{-k-1} \sum_{i=0}^\infty p^{-i}\\
&= \frac{p-1}{p^{k+1}} \frac{p}{p-1} = p^{-k}.
\end{align*}
(ii) Let $x\in \mathbb{Z}_p$ with
$$x = \sum_{i=0}^\infty a_i p^i$$
such that
$$\varphi(x) = \sum_{i=0}^\infty a_i p^{-i-1} \leq p^{-k}.$$
By definition $|\varphi(x)| \leq p^{-k}$ implies that $a_i = 0$ for $i < k-2$ and either $a_{k-1} = 0$ or $a_{k-1} = 1$ and $a_{l} = 0$ for all $l \geq k$.\\[12pt]
(iii) The claim follows again by definition of the Monna map and the limit of the geometric series.
\end{proof}
According to (iii) the preimage of $\varphi$ is well-defined if $a_i \neq p-1$ for infinitely many $i$. If this condition is violated we have a degree of freedom, which is the reason why we needed to make a choice in the definition in $\varphi^{-1}$ to get rid of the ambiguity described in (iii) of Lemma~\ref{lem:monna}. 
\begin{proof}[Proof of Theorem~\ref{thm:transfer}] 
Let $y_k = \varphi^{-1}(x_k)$ be a (real) sequence with $\alpha$-pair correlations and $s \geq 0$ be arbitrary. Therefore, we may without loss of generality assume $y_k \neq y_l$ for $k \neq l$ because the potential exceptions may not influence the limit of $F_{N,\alpha}(0)$. Then
$$|x_l - x_m|_p \leq \frac{s}{N^\alpha}$$
holds if and only if
$$|x_l - x_m|_p \leq p^{-k}$$
for the smallest $k \in \mathbb{N}_0$ such that $p^{-k} \leq \frac{s}{N^\alpha}$. By the definition of the inverse Monna map $\varphi^{-1}$ and Lemma~\ref{lem:monna} (ii) this holds if and only if 
$$|\varphi^{-1}(x_l)-\varphi^{-1}(x_m)| \leq p^{-k}.$$
or 
$$\varphi^{-1}(x_l) - \varphi^{-1}(x_m) = p^{-k-1}.$$
Since the elements of the sequence are distinct, for each $x_l$ there are at most two elements $x_m,x_{m'}$ with  $m,m' \in \mathbb{N}$ satisfying the latter condition. If $\alpha < 1$, then
$\frac{1}{N^2} \frac{1}{\mu(D_p(0,s/N^{\alpha}))} 2N \to 0$, i.e., the proportion of these elements is negligible in comparison to the prefactor, and hence we may without loss of generality assume that $|\varphi^{-1}(x_l) - \varphi^{-1}(x_m)| \leq p^{-k}$ for all $l,m \in \mathbb{N}$. For $\alpha = 1$ a different argument is required: let $s_0=1$ and consider the subsequence $N_k= \frac{1}{p^k}$. If 
$$\frac{1}{N_k} \# \{ 1 \leq l \neq m \leq N_k \, : \, \norm{y_l - y_m} = p^{-k}\} \to c > 0$$ then for for $s = 1 - c/4$ the limit must satisfy
$$\frac{1}{N_k} \# \{ 1 \leq l \neq m \leq N_k \, : \, \norm{y_l - y_m} \leq s \} \leq 2-c<2s$$
and we arrived at a contradiction. Hence we may as well assume without loss of generality that
\begin{align*} \# & \left\{ 1 \leq l \neq m \leq N : |x_l-x_m|_p \leq p^{-k} \right\}\\ & = \# \left\{ 1 \leq l \neq m \leq N : |y_l-y_m| \leq p^{-k} \right\}
\end{align*}
It follows that
\begin{align*}
\frac{1}{N^2} & \frac{1}{\mu(D_p(0,s/N^{\alpha}))} \# \left\{ 1 \leq l \neq m \leq N : |x_l-x_m|_p \leq \frac{s}{N^\alpha} \right\} \\
& = \frac{1}{N^2} \frac{1}{\mu(D_p(0,s/N^{\alpha}))} \# \left\{ 1 \leq l \neq m \leq N : |x_l-x_m|_p \leq p^{-k} \right\}\\
& = \frac{1}{N^2} p^k \# \left\{ 1 \leq l \neq m \leq N : |y_l-y_m| \leq p^{-k} \right\}.
\end{align*}
We have $\frac{s}{N^\alpha} = p^{-k}x$, where $x$ depends on $s$ and $N$ and $\frac{1}{p} < x = x(s,N) \leq 1$. Thus we can rewrite $F_{N,\alpha,p}(s)$ as
\begin{align*} \frac{1}{N^2} & \frac{N^\alpha}{\frac{s}{x}} \# \left\{ 1 \leq l \neq m \leq N : |y_l-y_m| \leq \frac{\frac{s}{x}}{N^\alpha}\right\}\\ & = \frac{1}{N^{2-\alpha}}\frac{1}{\frac{s}{x}} \# \left\{ 1 \leq l \neq m \leq N : |y_l-y_m| \leq \frac{\frac{s}{x}}{N^\alpha}\right\}
\end{align*}
Since $(y_n)_{n \in \mathbb{N}}$ has Poissonian pair correlations, it follows that
$$\frac{1}{N^{2-\alpha}} \# \left\{ 1 \leq l \neq m \leq N : |y_l-y_m| \leq \frac{s}{N^\alpha}\right\} \to s$$
for all $s\geq0$. Here the absolute value appears instead of the distance from the nearest integer $\norm{\cdot}$. For a formal proof of this well-known fact in the case $\alpha = 1$ we refer the reader to \cite{HZ23}, Proposition A, and it also holds for $\alpha < 1$. Therefore, it is enough to show that the convergence is uniform for $s$ in a compact interval. This suffices because for $\varepsilon > 0$ arbitrary it then holds that 
$$ \left| \frac{1}{N^{2-\alpha}}\frac{1}{\frac{s}{x}} \# \left\{ 1 \leq l \neq m \leq N : |y_l-y_m| \leq \frac{\frac{s}{x}}{N^\alpha}\right\} - 1\right| < \varepsilon$$
for all $N \geq N_0$ independent of $x$. Observing that $x \in [\tfrac{1}{p},1]$, a compact interval, this would complete the proof. The missing statement can be proven by the following standard result, which may be regarded as a variant of Dini's Theorem, compare \cite{Rud76}, Theorem 7.13.
\begin{lemma} \label{lem:convergence}
    Let $f_N:[a,b] \to \mathbb{R}$ be a sequence of functions such that $f_N$ is non-decreasing for every $N \in \mathbb{N}$. If $f_N$ converges pointwise to a continuous function $f:[a,b] \to \mathbb{R}$, then the convergence is uniform.
\end{lemma}
All the assertions from Lemma~\ref{lem:convergence} are indeed fulfilled: first, the function $F_{N,\alpha}^*(s):= \frac{1}{N^{2-\alpha}} \# \left\{ 1 \leq l \neq m \leq N : |y_l-y_m| \leq \frac{s}{N^\alpha}\right\}$ is non-decreasing in $s$ for all $N \in \mathbb{N}$. Second, we know that $F_{N,\alpha}^*(s) \to s$ pointwise, because $(y_n)_{n \in \mathbb{N}}$ is Poissonian, i.e., the limit is a continuous function. Hence the convergence is uniform on any compact interval, which finishes our proof.\\[12pt]
\end{proof}
For the sake of completeness we also include the proof of Lemma~\ref{lem:convergence}
\begin{proof} Let $f$ be the pointwise limit and $\varepsilon > 0$. Since $f$ is continuous, we may choose $a = t_0 < t_1 < \ldots < t_n = b$ such that $|f(x)-f(y)|< \varepsilon$ for all $x,y \in I_i := [t_{i-1},t_{i}]$ and all $i = 1, \ldots, n$. By the pointwise convergence there exists $N_0 \in \mathbb{N}$ such that
$$|f_N(t_i)-f(t_i)| < \varepsilon$$
for all $i = 0, \ldots, n$ and all $N \geq N_0$. As the $f_n$ are non-decreasing it thus holds for these $N$ that
$$f(t_{i-1}) - \varepsilon < f_N(t_{i-1}) < f_N(t) < f_N(t_i) < f(t_i) + \varepsilon$$
for all $t \in I_i$. By the choice of the intervals $I_i$ it follows that
$$f(t) - 2 \varepsilon < f_N(t) < f(t) + 2 \varepsilon$$
for all $t \in [a,b]$ which yields the claim.
\end{proof}


\bibliographystyle{alpha}
\bibdata{literatur}
\bibliography{literatur}

\newcommand{\etalchar}[1]{$^{#1}$}
\begin{thebibliography}{EBMV15}

\bibitem[EBMV15]{BMV15}
D.~El-Baz, J.~Marklof, and I.~Vinogradov.
\newblock The two-point correlation function of the fractional parts of $\sqrt{n}$ is {P}oisson.
\newblock {\em Proceeding of the AMS}, 143 (7):2815--2828, 2015.

\bibitem[HKL{\etalchar{+}}19]{HKL19}
A.~Hinrichs, L.~Kaltenb\"ock, G.~Larcher, W.~Stockinger, and M.~Ulrich.
\newblock On a multi-dimensional {P}oissonian pair correlation concept and uniform distribution.
\newblock {\em Monatshefte f\"ur Mathematik}, 190:333--352, 2019.

\bibitem[HZ23]{HZ23}
M.~Hauke and A.~Zafeiropoulos.
\newblock Poissonian correlations of higher orders.
\newblock {\em Journal of Number Theory}, 243:202--240, 2023.

\bibitem[LST24]{LST24}
C.~Lutsko, A.~Sourmelidis, and N.~Technau.
\newblock Pair correlation of the fractional parts of $\alpha n^{\theta}$.
\newblock {\em Journal of the European Mathematical Society, published online first}, 2024.

\bibitem[Mei67]{Mei67}
H.~G. Meijer.
\newblock Uniform distribution of g-adic numbers.
\newblock {\em Indagationes Mathematicae}, 29:535--546, 1967.

\bibitem[Mon52]{Mon52}
A.~Monna.
\newblock Sur une transformation simple des nombres p-adiques en nombres reels.
\newblock {\em Indagationes Mathematicae}, 14:1--9, 1952.

\bibitem[Neu99]{Neu99}
J.~Neukirch.
\newblock {\em Algebraic Number Theory}.
\newblock Springer, 1999.

\bibitem[RS98]{RS98}
Z.~Rudnick and P.~Sarnak.
\newblock The pair correlation function of fractional parts of polynomials.
\newblock {\em Communication in Mathematical Physics}, 194:61--70, 1998.

\bibitem[Rud76]{Rud76}
W.~Rudin.
\newblock {\em Principles of Mathematical Analysis}.
\newblock McGraw–Hill, 1976.

\bibitem[Wei22]{Wei22a}
C.~Wei\ss{}.
\newblock Some connections between discrepancy, finite gap properties and pair correlations.
\newblock {\em Monatshefte f\"ur Mathematik}, 199:909--927, 2022.

\bibitem[Wei23]{Wei23}
C.~Wei\ss{}.
\newblock A p-adic {P}oissoian pair correlations concept.
\newblock {\em arXiv:2308.15446}, 2023.

\bibitem[Wei24]{Wei24}
C.~Wei\ss{}.
\newblock Polynomial p-adic low-discrepancy sequences.
\newblock {\em arXiv:2406.09114}, 2024.

\end{thebibliography}

\textsc{Ruhr West University of Applied Sciences, Duisburger Str. 100, D-45479 M\"ulheim an der Ruhr,} \texttt{christian.weiss@hs-ruhrwest.de}

\end{document}